\documentclass[english]{article}
\usepackage[T1]{fontenc}
\usepackage[utf8]{inputenc}
\usepackage{xcolor}
\usepackage{amsthm}
\usepackage{amsmath}
\usepackage{amssymb}
\PassOptionsToPackage{normalem}{ulem}
\usepackage{ulem}

\makeatletter

\providecolor{lyxadded}{rgb}{0,0,1}
\providecolor{lyxdeleted}{rgb}{1,0,0}

\numberwithin{equation}{section}
\theoremstyle{plain}
\newtheorem{thm}{\protect\theoremname}[section]
\theoremstyle{plain}
\newtheorem{cor}[thm]{\protect\corollaryname}
\theoremstyle{plain}
\newtheorem{lem}[thm]{\protect\lemmaname}
\ifx\proof\undefined
\newenvironment{proof}[1][\protect\proofname]{\par
\normalfont\topsep6\p@\@plus6\p@\relax
\trivlist
\itemindent\parindent
\item[\hskip\labelsep
\scshape
#1]\ignorespaces
}{%
\endtrivlist\@endpefalse
}
\providecommand{\proofname}{Proof}
\fi
\theoremstyle{remark}
\newtheorem{rem}[thm]{\protect\remarkname}
\newcommand{\lyxaddress}[1]{
\par {\raggedright #1
\vspace{1.4em}
\noindent\par}
}

\@ifundefined{date}{}{\date{}}
\makeatother

\usepackage{babel}
\providecommand{\corollaryname}{Corollary}
\providecommand{\lemmaname}{Lemma}
\providecommand{\remarkname}{Remark}
\providecommand{\theoremname}{Theorem}

\begin{document}

\title{Harmonic Forms on Manifolds with Non-Negative Bakry-Émery-Ricci Curvature}

\author{Matheus Vieira}
\maketitle
\begin{abstract}
In this paper we prove that on a complete smooth metric measure space
with non-negative Bakry-Émery-Ricci curvature if the space of weighted
$L^{2}$ harmonic one-forms is non-trivial then the weighted volume
of the manifold is finite and universal cover of the manifold splits
isometrically as the product of the real line with an hypersurface.
\end{abstract}

\section{Introduction}

The theory of $L^{2}$ harmonic forms together with harmonic functions
has been used to study the geometry and topology of complete manifolds
(for example in \cite{LT1992,CSZ1997,LW2001,LW2002,CCZ2008}). The
theory of smooth metric measure spaces has also been attracting some
interest due to its connection with the Ricci flow (see \cite{C2010})
and as an independent research topic: there are works about volume
estimates \cite{WW2009,CZ2013}, essential spectrum of the drifted
Laplacian \cite{S2013}, etc. In \cite{L2003} Lott discussed (among
other things) the topology of compact smooth metric measure spaces
with non-negative Bakry-Émery-Ricci curvature using harmonic forms,
and in \cite{MW2011} Munteanu and Wang obtained geometrical and topological
results when these manifolds are non-compact using harmonic functions
(see also \cite{MW2012,MW2012-2}). In this paper we investigate harmonic
forms on non-compact smooth metric measure spaces with non-negative
Bakry-Émery-Ricci curvature.

Recall that a smooth metric measure space $(M,g,e^{-f}dv)$ is a Riemannian
manifold $(M,g)$ together with a smooth function $f$ and a measure
$e^{-f}dv$. The Bakry-Émery-Ricci curvature is defined by the formula
\[
\mbox{Ric}_{f}=\mbox{Ric}+\mbox{Hess}f.
\]
A differential form $\omega$ is called an $L_{f}^{2}$ differential
form if 
\[
\int_{M}|\omega|^{2}e^{-f}dv<\infty.
\]
It is well known that the formal adjoint of the exterior derivative
$d$ with respect to the $L_{f}^{2}$ inner product is given by the
formula
\[
\delta_{f}=\delta+\iota_{\nabla f},
\]
where $\iota_{\nabla f}$ denotes the interior product with the vector
field $\nabla f$ (see \cite{B1999} or Lemma \ref{adjoint}). The
$f$-Hodge Laplacian is defined by the formula
\[
\Delta_{f}=-(d\delta_{f}+\delta_{f}d).
\]
The space of $L_{f}^{2}$ harmonic one-forms is the set of all $L_{f}^{2}$
one-forms $\omega$ satisfying the equation 
\[
\Delta_{f}\omega=0.
\]

In this article we prove the following result (Theorem \ref{theorem}).
\begin{thm}
Let $(M^{n},g,e^{-f}dv)$ be a complete non-compact smooth metric
measure space with non-negative Bakry-Émery-Ricci curvature. If the
space of $L_{f}^{2}$ harmonic one-forms is non-trivial then the weighted
volume of $M^{n}$ is finite, that is
\[
\textnormal{vol}_{f}(M^{n})=\int_{M^{n}}e^{-f}dv<\infty,
\]
and the universal covering splits isometrically as $\tilde{M}^{n}=\mathbb{R}\times N^{n-1}$.
\end{thm}
A corresponding result for compact manifolds is discussed in \cite{L2003}
(Theorem 1, item 3).

In particular we obtain the following result (Corollary \ref{corollary 1})
concerning the vanishing of $L_{f}^{2}$ harmonic forms when the Bakry-Émery-Ricci
curvature is non-negative and the first eigenvalue of the $f$-Laplacian
is positive, which applies to non-trivial gradient Ricci steady solitons
(see Corolary \ref{corollary 2}).
\begin{cor}
Let $(M,g,e^{-f}dv)$ be a complete non-compact smooth metric measure
space with non-negative Bakry-Émery-Ricci curvature. If the first
eigenvalue of the $f$-Laplacian is positive then the space of $L_{f}^{2}$
harmonic one-forms is trivial.
\end{cor}
We also obtain the following vanishing result (Corollary \ref{corollary 3})
for the space of $L_{f}^{2}$ harmonic forms when the function $f$
is bounded.
\begin{cor}
Let $(M,g,e^{-f}dv)$ be a complete non-compact smooth metric measure
space with non-negative Bakry-Émery-Ricci curvature. If the function
$f$ is bounded then the space of $L_{f}^{2}$ harmonic one-forms
 is trivial.
\end{cor}
In section 2 we prove that $L_{f}^{2}$ harmonic forms of any degree
are closed and co-closed, in section 3 we prove Bochner's formula
and Kato's inequality for $L_{f}^{2}$ harmonic one-forms, in section
4 we prove Theorem \ref{theorem} and some corollaries.

The author wishes to thank his thesis advisor and the referee for
many helpful suggestions. This work is part of the author's Ph.D.
thesis, written under the supervision of Detang Zhou at Universidade
Federal Fluminense.

\section{$L_{f}^{2}$ Harmonic Forms are Closed and Co-Closed}

In this section we extend to smooth metric measure spaces the known
fact that $L^{2}$ harmonic forms of any degree are closed and co-closed.

Recall that the dot product of differential forms (of the same degree)
on a Riemannian manifold with a volume element $dv$ is defined by
\[
(\omega\cdot\eta)dv=\omega\wedge*\eta.
\]
The $L_{f}^{2}$ inner product of $L_{f}^{2}$ differential forms
is defined by 
\begin{align*}
(\omega,\eta)_{L_{f}^{2}(M)} & =\int_{M}\omega\cdot\eta\, e^{-f}dv.
\end{align*}
Notice that the $L_{f}^{2}$ inner product is also defined between
any two differential forms if one of them is compactly supported.

First we prove for completeness the known fact that $\delta_{f}$
is the formal adjoint of $d$.
\begin{lem}
\label{adjoint}Let $\omega$ be a $(p-1)$-form and $\eta$ be a
$p$-form on a smooth metric measure space $(M,g,e^{-f}dv)$. If one
of these differential forms is compactly supported then the following
identity holds
\begin{equation}
(d\omega,\eta)_{L_{f}^{2}(M)}=(\omega,\delta_{f}\eta)_{L_{f}^{2}(M)}.\label{eq adjoint}
\end{equation}
\end{lem}
\begin{proof}
Assume one of these differential forms compactly supported, say $\omega$.
We have 
\begin{align*}
(d\omega,\eta)_{L_{f}^{2}(M)} & =\int_{M}d\omega\cdot\eta\, e^{-f}dv\\
 & =\int_{M}d\omega\wedge*\eta\, e^{-f}\\
 & =\int_{M}\left(d(e^{-f}\omega)\wedge*\eta+e^{-f}df\wedge\omega\wedge*\eta\right)\\
 & =\int_{M}d(e^{-f}\omega\wedge*\eta)+\int_{M}\omega\wedge\left((-1)^{p}d*\eta+(-1)^{p-1}df\wedge*\eta\right)e^{-f}.
\end{align*}
Using Stokes theorem and the identities (see \cite{KLZ2008} for a
list of formulas)
\[
(-1)^{p}d*\eta=*\delta\eta
\]
and
\begin{align*}
df\wedge*\eta & =(-1)^{p-1}*\iota_{\nabla f}\eta
\end{align*}
we get identity (\ref{eq adjoint}).
\end{proof}
Now we prove that $L_{f}^{2}$ harmonic forms of any degree are closed
and co-closed. Here we adapt from Carron \cite{C2007}.
\begin{lem}
\label{closed}Every $L_{f}^{2}$ harmonic form (of any degree) on
a complete smooth metric measure space $(M,g,e^{-f}dv)$ is closed
and co-closed. In other words, 
\begin{equation}
d\omega=0\label{eq closed 1}
\end{equation}
 and 
\begin{equation}
\delta_{f}\omega=0.\label{eq closed 2}
\end{equation}
\end{lem}
\begin{proof}
We can choose a smooth function $\phi$ on $M$ such that $\phi=1$
in $B_{R}$, $\phi=0$ in $M\setminus B_{2R}$ and $|\nabla\phi|\leq\frac{2}{R}$
in $B_{2R}\setminus B_{R}$. Here $B_{R}$ denotes a ball with center
in a fixed point and radius $R$. We have 
\begin{align}
|d(\phi\omega)|_{L_{f}^{2}(M)}^{2} & =|d\phi\wedge\omega|_{L_{f}^{2}(M)}^{2}+(d\phi^{2}\wedge\omega,d\omega)_{L_{f}^{2}(M)}+|\phi d\omega|_{L_{f}^{2}(M)}^{2}\nonumber \\
 & =|d\phi\wedge\omega|_{L_{f}^{2}(M)}^{2}+(d(\phi^{2}\omega),d\omega)_{L_{f}^{2}(M)}\nonumber \\
 & =|d\phi\wedge\omega|_{L_{f}^{2}(M)}^{2}+(\phi^{2}\omega,\delta_{f}d\omega)_{L_{f}^{2}(M)}.\label{eq closed 3}
\end{align}
Notice that we used Lemma \ref{adjoint} in equation (\ref{eq closed 3}).
We also have the identity
\begin{align}
\delta_{f}(\phi\omega) & =\delta(\phi\omega)+\iota_{\nabla f}(\phi\omega)\nonumber \\
 & =\phi\delta\omega-\iota_{\nabla\phi}\omega+\phi\iota_{\nabla f}\omega\nonumber \\
 & =\phi\delta_{f}\omega-\iota_{\nabla\phi}\omega,
\end{align}
so 
\begin{align}
|\delta_{f}(\phi\omega)|_{L_{f}^{2}(M)}^{2} & =|\iota_{\nabla\phi}\omega|_{L_{f}^{2}(M)}^{2}-(\iota_{\nabla\phi^{2}}\omega,\delta_{f}\omega)_{L_{f}^{2}(M)}+|\phi\delta_{f}\omega|_{L_{f}^{2}(M)}^{2}\nonumber \\
 & =|\iota_{\nabla\phi}\omega|_{L_{f}^{2}(M)}^{2}+(\delta_{f}(\phi^{2}\omega),\delta_{f}\omega)_{L_{f}^{2}(M)}\nonumber \\
 & =|\iota_{\nabla\phi}\omega|_{L_{f}^{2}(M)}^{2}+(\phi^{2}\omega,d\delta_{f}\omega)_{L_{f}^{2}(M)}.\label{eq closed 4}
\end{align}
Notice that we used Lemma \ref{adjoint} again in equation (\ref{eq closed 4}).
Using the identity
\[
|d\phi\wedge\omega|^{2}+|\iota_{\nabla\phi}\omega|^{2}=|\nabla\phi|^{2}|\omega|^{2}
\]
and the assumption $\omega$ is an $L_{f}^{2}$ harmonic form, adding
equations (\ref{eq closed 3}) and (\ref{eq closed 4}) we get
\begin{equation}
|d(\phi\omega)|_{L_{f}^{2}(M)}^{2}+|\delta_{f}(\phi\omega)|_{L_{f}^{2}(M)}^{2}=\int_{M}|\nabla\phi|^{2}|\omega|^{2}e^{-f}dv\label{eq closed 5}
\end{equation}
Since $\omega$ is an $L_{f}^{2}$ differential form sending $R\to\infty$
in equation (\ref{eq closed 5}) we conclude that 
\[
|d\omega|_{L_{f}^{2}(M)}^{2}+|\delta_{f}\omega|_{L_{f}^{2}(M)}^{2}=0,
\]
which implies identities (\ref{eq closed 1}) and (\ref{eq closed 2}).
\end{proof}

\section{Bochner's Formula and Kato's Inequality}

Let us recall classical Bochner's formula for one-forms
\[
\frac{1}{2}\Delta|\omega|^{2}=|\nabla\omega|^{2}+\Delta\omega\cdot\omega+\mbox{Ric}(\omega,\omega)
\]
(here we use the same notation to represent the dual of a one-form)
and Kato's inequality for $L^{2}$ harmonic one-forms
\[
|\nabla\omega|^{2}\geq\frac{n}{n-1}|\nabla|\omega||^{2}.
\]

In this section we extend these formulas to smooth metric measure
spaces.

First we extend Bochner's formula, which is also found in Lott's paper
\cite{L2003} (equation 2.10). We thank the referee for pointing out
this reference.
\begin{lem}
\label{bochner}Let $\omega$ be a one-form on a smooth metric measure
space $(M,g,e^{-f}dv)$. Then the following identity holds 
\begin{equation}
\frac{1}{2}\Delta_{f}|\omega|^{2}=|\nabla\omega|^{2}+\Delta_{f}\omega\cdot\omega+\textnormal{Ric}_{f}(\omega,\omega).\label{eq bochner 1}
\end{equation}
\end{lem}
\begin{proof}
By a simple computation we have 
\[
\Delta_{f}=\Delta-d\iota_{\nabla f}-\iota_{\nabla f}d,
\]
so using the classical Bochner's formula for one-forms we get 
\begin{align*}
\frac{1}{2}\Delta_{f}|\omega|^{2} & =\frac{1}{2}\Delta|\omega|^{2}-\frac{1}{2}\nabla f\cdot\nabla|\omega|^{2}\\
 & =|\nabla\omega|^{2}+\Delta\omega\cdot\omega+\mbox{Ric}(\omega,\omega)-\frac{1}{2}\nabla f\cdot\nabla|\omega|^{2}\\
 & =|\nabla\omega|^{2}+\Delta_{f}\omega\cdot\omega+\mbox{Ric}_{f}(\omega,\omega)\\
 & -\frac{1}{2}\nabla f\cdot\nabla|\omega|^{2}-\mbox{Hess}f(\omega,\omega)+d\iota_{\nabla f}\omega\cdot\omega+\iota_{\nabla f}d\omega\cdot\omega.
\end{align*}
It remains to prove that 
\begin{equation}
-\frac{1}{2}\nabla f\cdot\nabla|\omega|^{2}-\mbox{Hess}f(\omega,\omega)+d\iota_{\nabla f}\omega\cdot\omega+\iota_{\nabla f}d\omega\cdot\omega=0.\label{eq bochner 2}
\end{equation}
We can choose a local orthonormal basis $e_{1},\dots,e_{n}$ with
dual basis $\theta_{1},\dots,\theta_{n}$ and assume the connection
forms $\theta_{ij}$ vanish on a fixed point. Writing $\omega=\omega_{i}\theta_{i}$
we have 
\begin{equation}
\frac{1}{2}\nabla f\cdot\nabla|\omega|^{2}=f_{i}\omega_{j}\omega_{ji}\label{eq bochner 3}
\end{equation}
and 
\begin{equation}
\mbox{Hess}f(\omega,\omega)=\omega_{i}\omega_{j}f_{ij}.\label{eq bochner 4}
\end{equation}
Computing 
\begin{align*}
d\iota_{\nabla f}\omega+\iota_{\nabla f}d\omega & =d(f_{i}\omega_{i})+i_{\nabla f}(\omega_{ij}\theta_{j}\wedge\theta_{i})\\
 & =\omega_{i}df_{i}+f_{i}d\omega_{i}+\omega_{ij}f_{j}\theta_{i}-\omega_{ij}f_{i}\theta_{j}\\
 & =\omega_{i}f_{ij}\theta_{j}+f_{i}\omega_{ij}\theta_{j}+\omega_{ij}f_{j}\theta_{i}-\omega_{ij}f_{i}\theta_{j}\\
 & =(\omega_{i}f_{ij}+f_{i}\omega_{ji})\theta_{j},
\end{align*}
we obtain 
\begin{equation}
d\iota_{\nabla f}\omega\cdot\omega+\iota_{\nabla f}d\omega\cdot\omega=\omega_{i}\omega_{j}f_{ij}+f_{i}\omega_{j}\omega_{ji}.\label{eq bochner 5}
\end{equation}
Equations (\ref{eq bochner 3}), (\ref{eq bochner 4}) and (\ref{eq bochner 5})
imply equation (\ref{eq bochner 2}).
\end{proof}
Now we extend Kato's inequality.
\begin{lem}
\label{kato}Let $\omega$ be an $L_{f}^{2}$ harmonic one-form on
a smooth metric measure space $(M^{n},g,e^{-f}dv)$. Then the following
inequality holds 
\begin{equation}
|\nabla\omega|^{2}\geq\frac{1}{n-1}\left(|\nabla|\omega||-|\nabla f\cdot\omega|\right)^{2}+|\nabla|\omega||^{2}.\label{eq kato}
\end{equation}
Moreover, if equality holds in (\ref{eq kato}) then
\begin{equation}
\nabla\omega=\begin{pmatrix}\lambda_{1} & 0 & \dots & 0 & 0\\
0 & \lambda_{2} & \dots & 0 & 0\\
\vdots & \vdots & \ddots & \vdots & \vdots\\
0 & 0 & \dots & \lambda_{2} & 0\\
0 & 0 & \dots & 0 & \lambda_{2}
\end{pmatrix}\label{eq kato matrix}
\end{equation}
where $\lambda_{1}=\nabla f\cdot\omega-(n-1)\lambda_{2}$.\end{lem}
\begin{proof}
We can choose a local orthonormal basis $e_{1},\dots,e_{n}$ with
dual basis $\theta_{1},\dots,\theta_{n}$. Writing 
\[
\omega=\sum_{i=1}^{n}\omega_{i}\theta_{i}
\]
we have 
\[
d\omega=\sum_{i,j=1}^{n}\omega_{ij}\theta_{j}\wedge\theta_{i}
\]
and
\[
\delta_{f}\omega=-\sum_{i=1}^{n}\omega_{ii}+\sum_{i=1}^{n}\omega_{i}f_{i}.
\]
By Lemma \ref{closed} we know that $L_{f}^{2}$ harmonic forms are
closed and co-closed therefore 
\begin{equation}
\omega_{ij}=\omega_{ji}\label{eq kato 1}
\end{equation}
for $i,j=1,\dots,n$ and 
\begin{equation}
\sum_{i=1}^{n}\omega_{ii}=\nabla f\cdot\omega.\label{eq kato 2}
\end{equation}
In the rest of the proof we adapt from Li and Wang \cite{LW2001}.
We can set $e_{1}=\frac{\omega}{|\omega|}$. Computing 
\begin{align*}
|\nabla|\omega|^{2}|^{2} & =4\sum_{i=1}^{n}\left(\omega_{1}\omega_{1i}\right)^{2}\\
 & =4|\omega|^{2}\sum_{i=1}^{n}(\omega_{1i})^{2},
\end{align*}
we obtain 
\begin{equation}
|\nabla|\omega||^{2}=\sum_{j=1}^{n}\omega_{1j}^{2}.\label{eq kato 3}
\end{equation}
Using identities (\ref{eq kato 1}), (\ref{eq kato 2}) and (\ref{eq kato 3})
we get 
\begin{align*}
|\nabla\omega|^{2} & =\omega_{11}^{2}+\sum_{j=2}^{n}\omega_{1j}^{2}+\sum_{j=2}^{n}\omega_{j1}^{2}+\sum_{i=2}^{n}\omega_{ii}^{2}+\sum_{i,j=2,i\neq j}^{n}\omega_{ij}^{2}\\
 & \geq\omega_{11}^{2}+2\sum_{j=2}^{n}\omega_{1j}^{2}+\frac{1}{n-1}(\sum_{i=2}^{n}w_{ii})^{2}\\
 & =\omega_{11}^{2}+2\sum_{j=2}^{n}\omega_{1j}^{2}+\frac{1}{n-1}\left(-\omega_{11}+\nabla f\cdot\omega\right)^{2}\\
 & =\frac{n}{n-1}\omega_{11}^{2}+2\sum_{j=2}^{n}\omega_{1j}^{2}+\frac{1}{n-1}(\nabla f\cdot\omega){}^{2}-\frac{2}{n-1}\omega_{11}\nabla f\cdot\omega\\
 & \geq\frac{n}{n-1}|\nabla|\omega||^{2}+\frac{1}{n-1}(\nabla f\cdot\omega)^{2}-\frac{2}{n-1}|\nabla|\omega||\cdot|\nabla f\cdot\omega|\\
 & =\frac{1}{n-1}\left(|\nabla|\omega||-|\nabla f\cdot\omega|\right)^{2}+|\nabla|\omega||^{2}.
\end{align*}
Matrix (\ref{eq kato matrix}) follows by analyzing intermediate inequalities.\end{proof}
\begin{rem}
The following inequality holds under the same assumptions of Lemma
 \ref{kato}:
\begin{equation}
|\nabla\omega|^{2}\geq\frac{n}{n-1}|\nabla|\omega||^{2}+\frac{1}{n-1}(\nabla f\cdot\omega)^{2}-\frac{2}{n-1}|\omega|^{2}(\nabla f\cdot\omega)\nabla\omega(\omega,\omega).
\end{equation}

\end{rem}

\section{Non-Negative Bakry-Émery-Ricci Curvature}

In this section we prove some results concerning smooth metric measure
spaces with non-negative Bakry-Émery-Ricci curvature.

First we prove the following theorem.
\begin{thm}
\label{theorem}Let $(M^{n},g,e^{-f}dv)$ be a complete non-compact
smooth metric measure space with non-negative Bakry-Émery-Ricci curvature.
If the space of $L_{f}^{2}$ harmonic one-forms is non-trivial then
the weighted volume of $M^{n}$ is finite, that is
\[
\textnormal{vol}_{f}(M^{n})=\int_{M^{n}}e^{-f}dv<\infty,
\]
and the universal covering splits isometrically as $\tilde{M}^{n}=\mathbb{R}\times N^{n-1}$.\end{thm}
\begin{proof}
Fix a non-trivial $L_{f}^{2}$ harmonic one-form $\omega$ on $M$.
Since the Bakry-Émery-Ricci curvature is non-negative and the one-form
$\omega$ is $f$-harmonic by Lemma \ref{bochner} we have
\[
\frac{1}{2}\Delta_{f}|\omega|^{2}\geq|\nabla\omega|^{2},
\]
so
\[
|\omega|\Delta_{f}|\omega|\geq|\nabla\omega|^{2}-|\nabla|\omega||^{2},
\]
hence by lemma \ref{kato} we have
\begin{equation}
|\omega|\Delta_{f}|\omega|\geq0.\label{eq theorem 1}
\end{equation}
We can choose a smooth function $\phi$ on $M$ such that $\phi=1$
in $B_{R}$, $\phi=0$ in $M\setminus B_{2R}$ and $|\nabla\phi|\leq\frac{2}{R}$
in $B_{2R}\setminus B_{R}$. Here $B_{R}$ denotes a ball with center
in a fixed point and radius $R$. Multiplying inequality (\ref{eq theorem 1})
by $\phi^{2}$ and integrating by parts we get 
\[
\int_{M}|\nabla|\omega||^{2}\phi^{2}e^{-f}dv\leq-2\int_{M}|\omega|\phi\nabla|\omega|\cdot\nabla\phi e^{-f}dv.
\]
Using Cauchy-Schuwarz and Young's inequalities we get 
\[
\frac{1}{2}\int_{M}|\nabla|\omega||^{2}\phi^{2}e^{-f}dv\leq2\int_{M}|\omega|^{2}|\nabla\phi|^{2}e^{-f}dv.
\]
Since $\omega$ is an $L_{f}^{2}$ differental one-form by the monotone
convergence theorem sending $R\to\infty$ we get 
\[
\frac{1}{2}\int_{M}|\nabla|\omega||^{2}e^{-f}dv\leq0,
\]
so the function $|\omega|$ is a constant $c$. Since 
\[
\int_{M}|\omega|^{2}e^{-f}dv=c^{2}\mbox{vol}_{f}(M)<\infty
\]
and $c\neq0$ it follows that the weighted volume $\mbox{vol}_{f}(M)$
is finite. Now we claim that $\omega$ is parallel. Indeed, since
$|\omega|$ is constant by Bochner's formula (Lemma \ref{bochner})
and the assumption in the Bakry-Émery-Ricci curvature we have 
\[
0\geq|\nabla\omega|^{2},
\]
so $\omega$ is parallel. Now the lifting of $\omega$ to the universal
cover $\tilde{M}$ is a non-trivial parallel one-form, which concludes
the proof by the de-Rham decomposition theorem.
\end{proof}
In particular we obtain the following result.
\begin{cor}
\label{corollary 1}Let $(M,g,e^{-f}dv)$ be a complete non-compact
smooth metric measure space with non-negative Bakry-Émery-Ricci curvature.
If the first eigenvalue of the $f$-Laplacian is positive then the
space of $L_{f}^{2}$ harmonic one-forms is trivial.\end{cor}
\begin{proof}
By assumption the first eigenvalue of the $f$-Laplacian
\[
\lambda_{1}(\Delta_{f})=\inf_{\phi\in C_{c}^{\infty}(M)}\frac{\int_{M}|\nabla\phi|^{2}e^{-f}dv}{\int_{M}\phi^{2}e^{-f}dv}
\]
is positive. We can choose a smooth function $\phi$ on $M$ such
that $\phi=1$ in $B_{R}$, $\phi=0$ in $M\setminus B_{2R}$ and
$|\nabla\phi|\leq\frac{2}{R}$ in $B_{2R}\setminus B_{R}$. Then
\[
\lambda_{1}(\Delta_{f})\int_{M}\phi^{2}e^{-f}dv\leq\frac{4}{R^{2}}\mbox{vol}_{f}(M).
\]
We claim that $\mbox{vol}_{f}(M)$ is not finite. Otherwise by the
monotone convergence sending $R\to\infty$ we get $\lambda_{1}(\Delta_{f})=0$,
a contradiction. Now since $\mbox{vol}_{f}(M)=\infty$ it follows
from Theorem \ref{theorem} that the space of $L_{f}^{2}$ harmonic
forms is trivial.
\end{proof}
Recall that a gradient steady Ricci soliton is a manifold $(M,g)$
together with a smooth function $f$ satisfying
\[
\mbox{Ric}+\mbox{Hess}f=0.
\]
In this case, it is possible to prove that the scalar curvature $R$
is non-negative and there is a constant $a$ such that
\[
R+|\nabla f|^{2}=a
\]
(see \cite{C2010}). It follows that $|\nabla f|\leq\sqrt{a}$, so
the only \emph{non-trivial} gradient steady Ricci solitons are those
with $a>0$, which must be non-compact (see \cite{C2010}). It was
proved by Munteanu and Wang in \cite{MW2011} that the first eigenvalue
of the $f$-Laplacian on non-trivial gradient steady Ricci solitions
is positive, more precisely $\lambda_{1}(\Delta_{f})=a^{2}/4$, so
Corollary \ref{corollary 1} implies the following result.
\begin{cor}
\label{corollary 2}On a complete non-compact non-trivial gradient
steady Ricci soliton $(M,g,f)$ the space of $L_{f}^{2}$ harmonic
one-forms is trivial.
\end{cor}
It was proved by Wei-Wyllie in \cite{WW2009} that on a complete non-compact
smooth metric measure space $(M,g,e^{-f}dv)$ if the function $f$
is bounded and the Bakry-Émery-Ricci curvature is non-negative then
the weighted volume $\mbox{vol}_{f}(M)$ is infinite, so Theorem \ref{theorem}
has the following consequence.
\begin{cor}
\label{corollary 3}Let $(M,g,e^{-f}dv)$ be a complete non-compact
smooth metric measure space with non-negative Bakry-Émery-Ricci curvature.
If the function $f$ is bounded then the space of $L_{f}^{2}$ harmonic
one-forms is trivial.\end{cor}

\lyxaddress{Departamento de Matemática, Universidade Federal do Espírito Santo,
29075-910, Vitória, Brazil}

\lyxaddress{matheus.vieira@ufes.br}
\end{document}